\documentclass[10pt,a4paper]{amsart}
\usepackage{amssymb,xcolor}
\usepackage{amsmath}
\title[Certain results on $\large{q}$-starlike error functions...]
{ Certain results on $\large{q}$-starlike and $\large{q}$-convex error functions }

\author[S. Kanas]{S. Kanas}

\address{University of Rzeszow, ul. S. Pigonia 1, 35-310 Rzeszow, Poland}

\email{skanas@ur.edu.pl}
\author[C. Ramachandran]{C. Ramachandran}

\author[L. Vanitha]{L. Vanitha}

\address{Department of Mathematics, University College of Engineering,
Villupuram, Anna University, Villupuram-605 103, Tamilnadu, India.}

\email{crjsp2004@yahoo.com}
\email{swarna.vanitha@gmail.com}

\newtheorem{theorem}{Theorem}[section]
\newtheorem{lemma}{Lemma}[section]

\newtheorem{definition}{Definition}[section]

\theoremstyle{remark}

\numberwithin{equation}{section}

\def \dsp {\displaystyle}
\begin{document}
\renewcommand{\thefootnote}{}



\begin{abstract}
The error function occurs widely in multiple areas of mathematics,
mathematical physics and natural sciences. There has been no work in
this area for the past four decades.  In this article, we estimate
the coefficient bounds with q-difference operator for certain
classes  of the spirallike starlike and convex error  function
associated with convolution product using subordination as well as
quasi-subordination. Though this concept is an untrodden path in the
field of complex function theory, it will prove to be an encouraging
future study for researchers on error function.
\end{abstract}

\maketitle

\section{Introduction and Preliminaries}\setcounter{equation}{0}
The \textit{error function} $erf$ defined by \cite[p. 297]{ma}
\begin{equation}\label{in1}
erf(z)  = \frac{2}{\sqrt{\pi}}\int_0^{z}exp(-t^{2})dt =
\frac{2}{\sqrt{\pi}}
\sum_{n=0}^\infty\frac{(-1)^{n}z^{2n+1}}{(2n+1)n!},
\end{equation}
is the subject of intensive studies and applications during the last
years. Several properties and inequalities of error function can be
found in \cite{ha, dc, fas, gmr}. In \cite{aeal, fc} the authors
study the properties of \textit{complementary error function}
occurring widely in almost every branch of applied mathematics and
mathematical physics, e.g., probability and statistics \cite{cqz}
and data analysis \cite{gh}. Its inverse, introduced by Carlitz
\cite{lc}, which we will denote by $inverf$, appears in multiple
areas of mathematics and the natural sciences. A few examples
include concentration-dependent diffusion problems \cite{jrp},
solutions to Einstein's scalar-field equations  and in heat
conduction problem  \cite{cqz, snk}.

Now, we recall the definitions of  fundamental class of analytic
functions. Let $\mathcal{A}$ denote the class of functions of the
form
\begin{equation}\label{in3}
f(z)=z+a_2 z^2 +\cdots,
\end{equation}
which are \textit{analytic} in the unit disk
$\mathbb{U}:=\{z\in\mathbb{C}:|z|<1\}$ and normalized by $f(0) = 0$,
 $f'(0) = 1$. Also, let $\mathcal{S}$ be the subclass of
$\mathcal{A}$ consisting of functions \textit{univalent} in
$\mathbb{U}$. Here and subsequently, $\Omega$ denotes the class of
analytic functions of the form
\begin{equation}\label{e1.7}
w(z) = w_1 z + w_2 z^2 + w_3 z^3 + \cdots,
\end{equation}
analytic  and satisfying a condition $|w(z)| < 1$ in $\mathbb{U}$,
known as a class of \textit{Schwarz functions}. To recall the
principle of subordination between analytic functions, let the
functions $f$ and $g$ be analytic in $\mathbb{U}$.  Then we say that
the function $f$ is \textit{subordinate} to $g$, if there exists a
Schwarz function $w$, such that $f(z) = g(w(z)) \ (z \in
\mathbb{U})$. We denote this subordination by $f \prec g $ (or $f(z)
\prec g(z), \ z \in \mathbb{U})$. In particular, if the function $g$
is univalent in $\mathbb{U}$, the above subordination is equivalent
to the conditions $f(0) = g(0), \  f(\mathbb{U}) \subset
g(\mathbb{U})$.

An extension of the notion of the subordination is the
quasi-subordination introduced by Robertson in \cite{msr}. We call a
function $f$ \textit{quasi-subordinate} to a function $g$ in
$\mathbb{U}$ if there exist  the Schwarz function $\omega$ and an
analytic functions $\varphi$ satisfying $|\varphi(z)| < 1$  such
that $f(z) = \varphi(z)g(w(z))$ in $\mathbb{U}$. We then write
$f\prec_{q} g$. If $\varphi(z) \equiv 1$ then the
quasi-subordination reduces to the subordination. If we set $w(z) =
z$, then $ f(z) = \varphi(z)g(z)$ and we say that $f$ is
\textit{majorized} by $g$ and it is written as $f(z)\ll g(z)$ in
$\mathbb{U}$. Therefore \textit{quasi-subordination} is a
generalization of the notion of the subordination as well as the
\textit{majorization} that underline its importance.  Related works
of quasi-subordination may be  found in \cite{mhmd, bsk} .

For $f$ given by \eqref{in3} and $g$ with the Taylor series $g(z) =
z + b_2 z^2+ \cdots$ their \textit{Hadamard product (or
convolution)}, denoted by $f \ast g$, is defined as
\begin{equation*}
(f \ast g)(z) = z + \sum_{n=2}^\infty a_n b_n z^n.
\end{equation*}
We will use also the symbol of the Hadamard product to the product
of the class as follows
\begin{equation}\label{had}
\mathcal{B}\ast\mathcal{C}=\{\psi\ast\phi: \psi\in \mathcal{B},
\phi\in \mathcal{C}\}.
\end{equation}

Let $E_{r}f$ be a normalized analytic function which is obtained
from \eqref{in1}, and given by
\begin{equation}\label{in2}
  E_{r}f(z) = \frac{\sqrt{\pi z}}{2}erf(\sqrt{z}) = z + \sum_{n=2}^\infty\frac{(-1)^{n-1}}{(2n-1)(n-1)!}z^{n}.
\end{equation}
Applying  a notation \eqref{had}  we define a family of an analytic
function as follows
\begin{equation}\label{in4}
\mathcal{E}=\mathcal{A}\ast E_{r}f=\left\{\mathcal{F}:\mathcal{F}(z)
= (f \ast E_{r}f)(z) = z +
\sum_{n=2}^\infty\frac{(-1)^{n-1}a_n}{(2n-1)(n-1)!}z^{n}, f\in
\mathcal{A}\right\},
 \end{equation}
where we denote by $E_{r}f$ the class that consists of a single
function $E_{r}f$.

Now, we refer  to a notion of \textit{q-operators} i.e.
\textit{q-difference operator} and \textit{q-integral operator} that
play vital role in the theory of hypergeometric series,  quantum
physics and in the operator theory. The application of q-calculus
was initiated by Jackson \cite{fhj, fhjj}. He was the first
mathematician who developed q-derivative and q-integral in a
systematic way. Mohammed and Darus \cite{md} studied approximation
and geometric properties of q-operators in some subclasses of
analytic functions in compact disk, Purohit and Raina \cite{prr},
Kanas and R\u{a}ducanu \cite{10} have used the fractional q-calculus
operators in investigations of certain classes of functions which
are analytic in the open disk, Purohit \cite{sdp} also studied these
q-operators defined by using convolution of normalized analytic
functions and q-hypergeometric functions. A comprehensive study on
applications of q-calculus in operator theory may be found in
\cite{aga}.  Both operators play crucial role in the theory of
relativity, usually encompasses two theories by  Einstein, one in
special relativity and the other in general relativity. Special
relativity applies to the elementary particles and their
interactions, whereas general relativity applies to the cosmological
and astrophysical realm, including astronomy. Special relativity
theory rapidly became a significant and necessary tool for theorists
and experimentalists in the new fields of atomic physics, nuclear
physics and quantum mechanics.

For $0<q<1$  the Jackson's \textit{q-derivative}  of a function
$f\in \mathcal{A}$ is, by definition,  given as follows \cite{fhj,
fhjj}
\begin{equation}\label{in5}
D_{q}f(z) = \left\{\begin{array}{lcl}\dfrac{f(z) - f(qz)}{(1 -
q)z}&for &z \neq 0,\\ f'(0) &for& z=0,\end{array}\right.
\end{equation}
and $D^{2}_{q}f(z) = D_{q}(D_{q}f(z)).$ From (\ref{in5}), we have
$$ D_{q}f(z) = 1 + \sum\limits_{n=2}^{\infty}[n]_{q}a_{n}z^{n-1},$$
where
\begin{equation}\label{in6}
[n]_{q} = \frac{1 - q^{n}}{1 - q},
\end{equation}
is sometimes called \textit{the basic number} $n$. If
$q\rightarrow1^{-}, [n]_{q}\rightarrow n$. For a function $h(z) =
z^{m},$ we obtain
$$D_{q}h(z) = D_{q}z^{m} = \frac{1 - q^{m}}{1 - q}z^{m-1} =
[m]_{q}z^{m-1},$$ and
$$\lim_{q\rightarrow 1}D_{q}h(z) = \lim_{q\rightarrow 1}\left([m]_{q}z^{m-1}\right) = mz^{m-1} = h'(z), $$
where $h'$ is the ordinary derivative.  Jackson q-derivative satisfy
known rules of differentiation, for example a q-analogue of Leibniz'
rule. As a right inverse, Jackson \cite{fhjj} introduced the
\textit{q-integral} of a function $f$
$$\int\limits_{0}^{z}f(t)d_{q}t =z(1 - q)\sum\limits_{n=0}^{\infty}q^nf(q^n z)= z(1 - q)\sum\limits_{n=0}^{\infty}a_nq^{n}z^{n},$$
provided that the series converges. For a function $h(z) = z^{m},$
we have
$$\int\limits_{0}^{z}h(t)d_{q}t = \int\limits_{0}^{z}t^{m}d_{q}t = \frac{z^{m+1}}{[m+1]_{q}}\quad (m \neq -1),$$
and
$$\lim_{q\rightarrow -1}\int\limits_{0}^{z}h(t)d_{q}t = \lim_{q\rightarrow -1}\frac{z^{m+1}}{[m+1]_{q}}
= \frac{z^{m+1}}{m+1} = \int\limits_{0}^{z}h(t)dt,$$
where $\int\limits_{0}^{z}h(t)dt$ is the ordinary integral.\\
As a consequence of \eqref{in5}, for $\mathcal{F}\in \mathcal{E}$ we
obtain
\begin{equation}\label{in7}
 D_{q}\mathcal{F}(z) = 1 + \sum_{n=2}^\infty\frac{(-1)^{n-1}[n]_{q}a_{n}}{(2n-1)(n-1)!}z^{n-1}.
\end{equation}
In the sequel we will use q-operators to the functions related to
the conic sections, that were introduced and studied by Kanas et al.
\cite{2} -- \cite{9} and examinated by several mathematicians in a
series of papers, see for example Ramachandran et al. \cite{crsa},
Kanas and R\u{a}ducanu \cite{10}, Sim et al. \cite{26}, etc.
Kharasani \cite{ALK} extended original definition to the p-valent
functions generalizing the domains $\Omega_k$ to $\Omega_{k,\alpha}$
( $0\leq k<\infty,\ 0\le \alpha < 1$) as follows:
\begin{equation*}
\Omega_{k,\alpha}=\{w=u+iv:(u-\alpha)^2>k^2(u-1)^2+k^2v^2\}, \
\Omega_{k,0}=\Omega_k.
\end{equation*} Various classes of functions were defined by the fact of the membership to the domain $\Omega_{k,\alpha}$, for instance
by setting  $w=p(z)=\dsp\frac{zf'(z)}{f(z)}$ or
$p(z)=1+\dfrac{zf''(z)}{f'(z)}$. We note that the explicit form of
function $p_{k,\eta}$ that maps the unit disk onto the domains
bounded by $\Omega_{k,\alpha}$ and such that $1\in
\Omega_{k,\alpha}$ is as follows
 $$p_{0,\alpha}(z)=\frac{1+(1-2\alpha)z}{1-z},\quad p_{1,\alpha}(z)=1+\frac{2(1-\alpha)}{\pi^2}\log^2\frac{1+\sqrt{z}}{1-\sqrt{z}},$$
\begin{equation*}
 p_{k,\alpha}(z)  = \frac{(1-\alpha)}{1-k^2} \cos \left(A(k)i \log \frac{1+ \sqrt{z}}{1-\sqrt{z}}\right) - \frac{k^2-\alpha}{1-k^2}\quad ( 0<k<1),
\end{equation*} and
$$p_{k,\alpha}(z)=\frac{(1-\alpha)}{k^2-1}\sin^2\left(\frac{\pi}{2\kappa(t)}\mathcal{K}
\left(\frac{\sqrt{z}}{\sqrt{t}},t\right)\right)+\frac{k^2-\alpha}{k^2-1}\quad
(k>1),$$ where $z\in \mathbb{U}$,  $A(k)=\frac{2}{\pi}\arccos k$ and
$\mathcal{K}(\omega,t)$ is the Legendre elliptic integral of the
first kind
$$\mathcal{K}(\omega,t)=\int_0^{\omega}\frac{dx}{\sqrt{1-x^2}\sqrt{1-t^2x^2}}\quad(\kappa(t)=\mathcal{K}(1,t)),$$
with $t\in(0,1)$ chosen such that
$k=\cosh\frac{\pi\kappa'(t)}{4\kappa(t)}$.

By virtue of the properties of the domains, for $p\prec
p_{k,\alpha}$, we have
\begin{equation}\label{1.4}
\Re p(z) \ge \Re(p_{k,\alpha}(z))> (k+\alpha)/(k+1).
\end{equation}
Note that Kanas and Sugawa \cite{kas} proved  the positivity of
coefficients of the functions $p_{k,0}$ that implies positivity of
$p_{k,\alpha}$ for $0\le \alpha<1$. Also, we note that the domains
$\Omega_{k,\alpha}$ are symmetric about real axis and starlike with
respect to $1$.

\begin{definition}Let $0\le k< \infty,\ 0\le \alpha <1,\ -\frac{\pi}{2}<\beta< \frac{\pi}{2},\
0<q<1,\ b\neq 0$, and let $p_{k,\alpha}(z)$ be defined as above.
 A function $f\in\mathcal{A}$ is in the class $\mathcal{ES}^{\beta}_{q,b}(p_{k,\alpha})$ if
\begin{equation}\label{d1.6}
 1+\frac{1}{b}\left( (1+i\tan \beta)\left(\frac{zD_{q}\mathcal{F}(z)}{\mathcal{F}(z)}\right)-i\tan \beta -1\right)\prec
 p_{k,\alpha}(z)\qquad (z\in \mathbb{U}).
\end{equation}
A function $f\in\mathcal{A}$ is in the class
$\mathcal{EC}^{\beta}_{q,b}(p_{k,\alpha})$ if
\begin{equation*}
 1+\frac{1}{b}\left((1+i\tan \beta)\left(\frac{\left(zD_{q}\mathcal{F}(z)\right)'}{D_{q}(\mathcal{F}(z))}\right)-i
 \tan \beta -1\right)\prec p_{k,\alpha}(z)\qquad (z\in \mathbb{U}).
\end{equation*}
\end{definition}

Let $\phi(z) = 1 + c_{1}z +c_{2}z^{2} + \cdots\ (c_1 >0)$ be  an
analytic function with positive real part on $\mathbb{U}$  which
maps the open  unit disk $\mathbb{U}$ onto a region starlike with
respect to $1$ and symmetric with respect to the real axis.

\begin{definition}Let $0\le k< \infty,\ 0\le \alpha <1,\ -\frac{\pi}{2}<\beta< \frac{\pi}{2},\
0<q<1,\ b\neq 0$. By $\widetilde{\mathcal{ES}}^{\beta}_{q,b}(\phi)$
we mean a family that consist of the functions $f\in\mathcal{A}$
satisfying the quasi-subordination
\begin{equation*}
 1+\frac{1}{b}\left( (1+i\tan \beta)\left(\frac{zD_{q}\mathcal{F}(z)}{\mathcal{F}(z)}\right)-i\tan \beta -1\right)\prec _{q}\phi(z)-1 \quad (z\in
 \mathbb{U}),
\end{equation*}
and let the class $\widetilde{\mathcal{EC}}^{\beta}_{q,b}(\phi)$
consist of the functions $f\in\mathcal{A}$ satisfying the
quasi-subordination
\begin{equation*}
 1+\frac{1}{b}\left((1+i\tan \beta)\left(\frac{\left(zD_{q}\mathcal{F}(z)\right)'}{D_{q}(\mathcal{F}(z))}\right)-i
 \tan \beta -1\right)\prec _{q}\phi(z)-1\quad (z\in \mathbb{U}).
\end{equation*}
\end{definition}
The principal significance of the sharp bounds of the coefficients
is the information about geometric properties of the functions. For
instance, the sharp bounds of the second coefficient of normalized
univalent functions readily yields the growth and distortion bounds.
Also, sharp bounds of the coefficient functional $\left| a_{3} - \mu
a^2_{2} \right|$ obviously help in the investigation of univalence
of analytic functions. Apart from these $n$-th coefficient bounds
were used to determine the extreme points of the classes of analytic
functions. Estimates of Fekete-Szeg\"{o} functional for various
subclasses of
univalent and multivalent functions were given, among other, in \cite{oma, ar, crks, crsa1}.\\
In this paper, we obtain coefficient estimates for the functions in the above defined class for
q-difference operator associated with subordination and quasi subordination.\\

The following lemma is needed to prove our main results. Lemma
\ref{2.14} is a reformulation of the corresponding result for
functions with positive real part due to Ma and Minda \cite{mm}.

\begin{lemma}\cite{ALI}\label{2.14}
If $w \in \Omega$, then
\begin{equation}\label{2.15}
\left| w_2 - t w_1^2 \right| \leq \left\{
\begin{array}{rcl}
-t & if & t < -1, \\
1  & if &  -1\le t\le 1, \\
t & if & t > 1.
\end{array} \right.
\end{equation}
When $t < -1$ or $t > 1$, the equality holds if and only if $w(z) = z$ or one of its rotations.
If $ -1 < t < 1$, then equality holds if and only if $w(z) = z^2$ or one of its rotations. Equality holds for $t = -1$ if and only if
\begin{equation}\label{2.16}
w(z) = z \frac{\lambda + z}{1 + \lambda z} \qquad (0 \le \lambda
\leq 1)
\end{equation}
or one of its rotations, while for $t = 1$ the equality holds if and
only if
\begin{equation}\label{2.17}
w(z) = -z \frac{\lambda + z}{1 + \lambda z} \qquad (0 \le \lambda
\leq 1)
\end{equation}
or one of its rotations.

Although the above upper bound is sharp, it can be improved in the
case, when $-1 < t < 1$
\begin{equation}\label{2.18}
\begin{array}{l c c}
\left| w_2 - t w_1^2 \right| + (1+t) \left|w_1\right|^2 \le 1 & & (-1 < t \leq 0), \\ \\
\left| w_2 - t w_1^2 \right| + (1-t) \left|w_1\right|^2 \le 1 & & (0
< t < 1).
\end{array}
\end{equation}
\end{lemma}

\section{The Fekete-Szeg\"{o} functional associated with conical sections}

In this section we will consider the behavior of the
Fekete-Szeg\"{o} functional defined on the classes related to the
conical domains

\begin{theorem}\label{th1}Let $\ 0\le k< \infty,\ 0\le \alpha <1$, and let
$p_{k,\alpha}(z)=1+p_1z+p_2z^2+\cdots$. Set
$$
\sigma_1
=\frac{10\left(p_1+p_2\right)\varrho\vartheta_2^{2}+10\vartheta_2 b
p_1^2}{9\vartheta_3 b p_1^2},  \quad \sigma_2  =
\frac{10\left(p_1-p_2\right)\varrho\vartheta_2^{2}-10\vartheta_2 b
p_1^2}{9\vartheta_3 b p_1^2}, $$ and$$\sigma_3  =
\frac{10\varrho\vartheta_2^{2}p_2+10\vartheta_2 b
p_1^2}{9\vartheta_3 b p_1^2}.$$ If $f$ given by \eqref{in3} belongs
to $\mathcal{ES}^{\beta}_{q,b}(p_{k,\alpha})$, then
\begin{equation}\label{4.1}
\left|a_{3} - \mu a^2_{2} \right| \leq \left\{
\begin{array}{l c l}
 \dsp\frac{10bp_1}{\varrho\vartheta_3}
 \left(\dsp\frac{p_2}{p_1}+\dsp\frac{9\mu\vartheta_3-10\vartheta_2}{10\varrho\vartheta_2^{2}}b p_1
    \right) & if & \mu < \sigma_1, \\ \\
 \dsp\frac{10bp_1}{\varrho\vartheta_3}   & if & \sigma_1 \leq \mu \leq \sigma_2, \\ \\
 \dsp\frac{10bp_1}{\varrho\vartheta_3}
 \left(\dsp\frac{p_2}{p_1}+\dsp\frac{10\vartheta_2-9\mu\vartheta_3}{\varrho\vartheta_2}b p_1
    \right) & if & \mu > \sigma_2,
\end{array} \right.
\end{equation}where $\varrho =1+i\tan\beta$, $\vartheta_2=[2]_{q}-1$ and
$\vartheta_3=[3]_{q}-1$. Further, if $\sigma_1 \leq \mu \leq
\sigma_3$, then
\begin{equation}\label{th4.2}
 \left| a_{3} - \mu a^2_{2} \right| + \dsp\frac{10\varrho\vartheta_2^{2}}{9\vartheta_3b p_1}
 \left(p_1+p_2+\dsp\frac{10\vartheta_2-9\mu\vartheta_3}{10\vartheta_2^{2}}
 b p_1^{2}\right)\left|a_{2}\right|^2 \leq
 \dsp\frac{10bp_1}{\varrho\vartheta_3},
\end{equation}
and, if $\sigma_3 \leq \mu \leq \sigma_2$, then
\begin{equation}\label{t2.3}
\left| a_{3} - \mu a^2_{2} \right| +
\dsp\frac{10\varrho\vartheta_2^{2}}{9\vartheta_3 b p_1}
\left(p_1-p_2-\dsp\frac{10\vartheta_2-9\mu\vartheta_3}{10\varrho\vartheta_2^{2}}b
p_1^{2}\right)\left|a_{2}\right|^2 \leq
\dsp\frac{10bp_1}{\varrho\vartheta_3}.
\end{equation}
For any complex number $\mu$,
\begin{equation}\label{4.2}\begin{split}
\left| a_{3} - \mu a^2_{2} \right| \leq & \dsp\frac{10bp_1}{\varrho
\vartheta_3}
\max\left\{1,\left|\dsp\frac{p_2}{p_1}+\dsp\frac{10\vartheta_2-9\mu\vartheta_3}{\varrho\vartheta_2}b
p_1\right|\right\}.
\end{split}
\end{equation}\end{theorem}

\begin{proof}
If $f \in \mathcal{ES}^{\beta}_{q,b}(p_{k,\alpha})$, then there is a
Schwarz function $w \in \Omega$ of the form \eqref{e1.7} such that
\begin{equation}\label{4.3}
1+\frac{1}{b}\left((1+i\tan\beta)
\left(\frac{zD_{q}\mathcal{F}(z)}{\mathcal{F}(z)}\right)-i\tan\beta-1\right)
= p_{k,\alpha}(w(z)).
\end{equation}
We note that
\begin{equation}\label{4.4}
\frac{z D_{q}\mathcal{F}(z)}{\mathcal{F}(z)}= 1 +
\frac{1-[2]_{q}}{3}a_{2} z + \left(\frac{[3]_{q}-1}{10}a_{3} +
\frac{1-[2]_{q}}{9}a_{2}^{2}\right) z^2 +  \dots
\end{equation} and
\begin{equation}\label{4.4a}
 p_{k,\alpha}(w(z))=1+p_1w_1+(p_1w_2+p_2w_1^2)z^2+(p_1w_3+2p_2w_1w_2+
 p_3w_1^3)z^3+\cdots .
\end{equation}
Applying \eqref{4.3}, \eqref{4.4} and \eqref{4.4a}, we obtain
\begin{equation}\label{4.5}
a_{2} =  \frac{3bp_1 w_1}{(1+i\tan\beta)(1-[2]_{q})},
\end{equation}
and
\begin{equation}\label{4.6}
a_{3} =  \frac{10bp_1}{([3]_{q}-1)(1+i\tan\beta)} \left(w_{2}
-\left(\frac{p_2}{p_1} +
\frac{p_1b}{(1+i\tan\beta)(1-[2]_{q})}\right)w_{1}^{2}\right).
\end{equation}
Hence, by \eqref{4.5}, \eqref{4.6}, we get the following
\begin{equation}\label{4.8}
a_{3}-\mu
a_{2}^2=\dsp\frac{10bp_1}{\left(1+i\tan\beta\right)\left([3]_{q}-1\right)}\left(w_2-t
w_1^2\right),
\end{equation}
where
\begin{equation*}
t
=\dsp\frac{p_2}{p_1}+\dsp\frac{10\left(1-[2]_{q}\right)-9\mu\left(1-[3]_{q}\right)}{10(1+i\tan\beta)\left(1-[2]_{q}\right)^{2}}b
p_1.
\end{equation*}
The results \eqref{4.1} -- \eqref{4.2} are established by an
application of Lemma \ref{2.14} and using the notation $\varrho
=1+i\tan\beta$, $\vartheta_2=[2]_{q}-1$ and $\vartheta_3=[3]_{q}-1$.
To show that the bounds in \eqref{4.1} -- \eqref{4.2} are sharp, we
define the function $g_{\phi n}$ $(n = 2, 3, \ldots)$ by
\begin{equation*}
1+\frac{1}{b}\left((1+i\tan\beta)\left(\frac{zD_{q}g_{\phi
n}(z)}{g_{\phi n}(z)}\right)-i\tan\beta-1\right) =
p_{k,\alpha}(z^{n-1}),
\end{equation*}with $g_{\phi n}(0) = 0 = g_{\phi n}'(0)-1,
$ and the function $h_\lambda$ and $k_\lambda$ $(0 \leq \lambda \leq
1)$ by
\begin{equation*}
1+\frac{1}{b}\left((1+i\tan\beta)\left(\frac{zD_{q}h_\lambda(z)}{h_\lambda(z)}
\right)-i\tan\beta-1\right) = p_{k,\alpha}\left(z \frac{\lambda +
z}{1 + \lambda z}\right),
\end{equation*}
\begin{equation*}
1+\frac{1}{b}\left(\frac{(1+i\tan\beta)}{p}
\left(\frac{zD_{q}k_\lambda(z)}{k_\lambda(z)}
\right)-i\tan\beta-1\right) = p_{k,\alpha}\left(-z \frac{\lambda +
z}{1 + \lambda z}\right).
\end{equation*}
with  $h_\lambda(0) = h_\lambda'(0)-1=k_\lambda(0) =
k_\lambda'(0)-1=0$. Clearly  $g_{\phi n}, \ h_\lambda, \ k_\lambda
\in \mathcal{ES}^{\beta}_{q,b}(p_{k,\alpha})$.  Also we set
$g_{\phi} : = g_{\phi 2}$.  If $\mu < \sigma_1$ or $\mu > \sigma_2$,
then the equality holds if and only if $f$ is $g_{\phi}$ or one of
its rotations.  When $\sigma_1 < \mu < \sigma_2$, then the equality
holds if and only if $f$ is $g_{\phi 3}$ or one of its rotations. If
$\mu = \sigma_1,$ then the equality holds if and only if $f$ is
$h_\lambda$ or one of its rotations.  If $\mu = \sigma_2,$ then the
equality holds if and only if $f$ is $k_\lambda$ or one of its
rotations.
\end{proof}

The following result may be proved in much the same way as Theorem
\ref{th1} (we also use a notation $\varrho =1+i\tan\beta$).

\begin{theorem}\label{th2}Let $\ 0\le k< \infty,\ 0\le \alpha <1$, and let
$p_{k,\alpha}(z)=1+p_1z+p_2z^2+\cdots$. For $f\in
\mathcal{EC}^{\beta}_{q,b}(p_{k,\alpha})$, it holds
\begin{equation*}
\left| a_{3} - \mu a^2_{2} \right| \leq \left\{
\begin{array}{l c l}
\dsp\frac{5bp_1}{\varrho[3]_{q}}
\left(\dsp\frac{p_2}{p_1}+\dsp\frac{bp_1}{\varrho}-\dsp\frac{9\mu[3]_{q}b
p_1}{5[2]_{q}^{2}\varrho}
\right) & if & \mu < \sigma_1,\\ \\
\dsp\frac{5bp_1}{\varrho[3]_{q}}  & if & \sigma_1 \leq \mu \leq \sigma_2, \\ \\
\dsp - \frac{5bp_1}{\varrho[3]_{q}}
\left(\dsp\frac{p_2}{p_1}+\dsp\frac{bp_1}{\varrho}-\frac{9\mu[3]_{q}b
p_1}{5[2]_{q}^{2}\varrho}
    \right) & if & \mu > \sigma_2,
\end{array} \right.
\end{equation*}where
$$\sigma_1  =\frac{5[2]_{q}^{2}\varrho}{9[3]_{q}b
p_1^2}\left(p_2-p_1+\frac{bp_1^2}{\varrho}\right), \quad \sigma_2  =
\frac{5[2]_{q}^{2}\varrho}{9[3]_{q}b
p_1^2}\left(p_1+p_2+\frac{bp_1^2}{\varrho}\right),$$ and $$\sigma_3
 =  \frac{5[2]_{q}^{2}\varrho}{9[3]_{q}b
p_1^2}\left(p_2+\frac{bp_1^2}{\varrho}\right).
$$
Further, if $\sigma_1 \leq \mu \leq \sigma_3$, then
\begin{equation*}
\left| a_{3} - \mu a^2_{2} \right| +
\dsp\frac{5[2]_{q}^{2}\varrho}{9\mu[3]_{q} b p_1}
\left(p_1-p_2-\frac{bp_1^2}{\varrho}+\dsp\frac{9\mu[3]_{q}b
p_1}{5[2]_{q}^{2}\varrho}\right)\left|a_{2}\right|^2 \leq
\dsp\frac{5bp_1}{\varrho[3]_{q}},
\end{equation*}

and, if $\sigma_3 \leq \mu \leq \sigma_2$, then
\begin{equation*}
\left| a_{3} - \mu a^2_{2} \right| +
\dsp\frac{5[2]_{q}^{2}\varrho}{9\mu[3]_{q}b p_1}
\left(p_1+p_2+\frac{bp_1^2}{\varrho}-\dsp\frac{9\mu[3]_{q}b
p_1}{5[2]_{q}^{2}\varrho}\right)\left|a_{2}\right|^2 \leq
\dsp\frac{5bp_1}{\varrho[3]_{q}}.
\end{equation*}

For any complex number $\mu$
\begin{equation*}
\begin{split}
\left| a_{3} - \mu a^2_{2} \right| \leq &
\dsp\frac{5bp_1}{\varrho[3]_{q}}\max
\left\{1,\left|\dsp\frac{p_2}{p_1}+\dsp\frac{bp_1}{\varrho}-\frac{9\mu[3]_{q}b
p_1}{5[2]_{q}^{2}\varrho}\right|    \right\}.
 \end{split}
\end{equation*}
\end{theorem}

\section{The Fekete-Szeg\"{o} functional associated with quasi-subordination}

We now direct our attention to the extension of the subordination
idea. The principal difference is the assertion of
quasi-subordination. We are thus led to the following strengthening
of Theorem \ref{th1} and \ref{th2}.

\begin{theorem}\label{3.1}Let $\ -\frac{\pi}{2}<\beta< \frac{\pi}{2},\
0<q<1$, $\ b\neq 0$ and let $\varrho =1+i\tan\beta$. If $f$ of the
form  \eqref{in3} belongs to
$\widetilde{\mathcal{ES}}^{\beta}_{q,b}(\phi)$, then
\begin{equation*}
\left|a_{2}\right| \leq \frac{3bc_{1}}{\varrho(1-[2]_{q})},
\end{equation*}
\begin{equation}\label{eq3.1}
\left|a_{3}\right| \leq
\frac{10b}{\varrho([3]_{q}-1)}\left(c_{1}+\max\left\{c_{1},
\left|\frac{ bc_{1}^{2}}{\varrho([2]_{q}-1)}\right|
+|c_{2}|\right\}\right),
\end{equation}
and for any complex number $\mu$,
\begin{equation}\label{eq3.2}
\left| a_{3} - \mu a^2_{2} \right| \leq
\frac{10b}{\varrho([3]_{q}-1)}\left(c_{1}+\max\left\{c_{1},
\left|\frac{10(1-[2]_{q})+9\mu
b([3]_{q}-1)}{10\varrho(1-[2]_{q})^{2}}\right|bc_{1}^{2}
+|c_{2}|\right\}\right).
\end{equation}
\end{theorem}
\begin{proof}
If $f \in \widetilde{\mathcal{ES}}^{\beta}_{q,b}(\phi)$,  then there
exist analytic functions $\varphi$ and $\omega$ with
$\left|\varphi(z)\right| \leq 1, \omega(0) = 0$ and
$\left|\omega(z)\right| < 1$ such, that
\begin{equation}\label{eq3.3}
1+\frac{1}{b}\left(
\varrho\left(\frac{zD_{q}\mathcal{F}(z)}{\mathcal{F}(z)}\right)-\varrho\right)
= \varphi(z)\left(\phi(\omega(z))-1\right).
\end{equation}
 Since
\begin{equation}\label{eq3.4}
\frac{zD_{q}\mathcal{F}(z)}{\mathcal{F}(z)}= 1 +
\frac{1-[2_{q}]}{3}a_{2} z + \left(\frac{[3]_{q}-1}{10}a_{3} +
\frac{1-[2_{q}]}{9}a_{2}^{2}\right) z^2 +  \dots
\end{equation}
\begin{equation}\label{eq3.5}
\varphi(z)\left(\phi(\omega(z))-1\right) = c_{1}d_{0}\omega_{1}z +
\left(c_{1}d_{1}\omega_{1} +
d_{0}(c_{1}\omega_{2}+c_{2}\omega_{1}^{2})\right)z^{2}+\dots
\end{equation}
From \eqref{eq3.3}, \eqref{eq3.4} and \eqref{eq3.5}, we get
\begin{equation*}
a_{2} = \frac{3bc_{1}d_{0}\omega_{1}}{\varrho(1-[2]_{q})}
\end{equation*}
\begin{equation*}
a_{3} = \frac{10b}{\varrho([3]_{q}-1)}\left(c_{1}d_{1}\omega_{1} +
c_{1}d_{0}\omega_{2} + d_{0} \left(c_{2} -
\frac{bc_{1}^{2}d_{0}}{\varrho(1-[2]_{q})}\right)\omega_{1}^{2}\right)
\rm{and}
\end{equation*}
\begin{multline*}
\left|a_{3} - \mu a^2_{2} \right| \leq\frac{10b}{\varrho([3]_{q}-1)}
\left[\left|c_{1}d_{1}\omega_{1}\right|+\left|c_{1}d_{0}\left\{\omega_{2}-
\left(\frac{bc_{1}d_{0}}{\varrho(1-[2]_{q})}\right.\right.\right.\right.\\
\left.\left.\left.\left. +\frac{9\mu
b^{2}c_{1}d_{0}([3]_{q}-1)}{10(1+i\tan\beta)(1-[2]_{q})^{2}}-\frac{c_{2}}{c_{1}}\right)\omega_{1}^{2}\right\}\right|\right].
\end{multline*}
Since $\varphi$ is analytic in $\mathbb{U}$, using the inequalities
$\left|d_{n}\right|\leq1$ and $\left|\omega_{1}\right|\leq1$, we get
\begin{equation*}
\left|a_{2}\right| \leq \frac{3bc_{1}}{\varrho(1-[2]_{q})}.
\end{equation*}
\begin{equation}\label{eq3.6}
\left|a_{3} - \mu a^2_{2} \right|
\leq\frac{10bc_{1}}{\varrho([3]_{q}-1)}\left\{1+\left|\omega_{2}-\left(-\frac{c_{2}}{c_{1}}-
\left[\frac{bc_{1}}{\varrho(1-[2]_{q})}+\frac{9\mu
b^{2}c_{1}([3]_{q}-1)}{10\varrho(1-[2]_{q})^{2}}\right]c_{1}\right)\omega_{1}^{2}\right|\right\}.
\end{equation}
Applying Lemma \ref{2.14} to
\begin{equation*}
  \left|\omega_{2}-\left(-\frac{c_{2}}{c_{1}}-\left[\frac{bc_{1}}{\varrho(1-[2]_{q})}
  +\frac{9\mu b^{2}c_{1}([3]_{q}-1)}{10\varrho(1-[2]_{q})^{2}}\right]c_{1}\right)\omega_{1}^{2}\right|,
\end{equation*}
which yields
\begin{equation*}
\left|a_{3} - \mu a^2_{2} \right| \leq
\frac{10b}{\varrho([3]_{q}-1)}\left(c_{1}+\max\left\{c_{1},
\left|\frac{10(1-[2]_{q})+9\mu
b([3]_{q}-1)}{10\varrho(1-[2]_{q})^{2}}\right|bc_{1}^{2}
+|c_{2}|\right\}\right).
\end{equation*}
For $\mu=0$, we get
\begin{equation*}
\left|a_{3}\right| \leq
\frac{10b}{\varrho([3]_{q}-1)}\left(c_{1}+\max\left\{c_{1},
\left|\frac{bc_{1}^{2}}{\varrho([2]_{q}-1)}\right|
+|c_{2}|\right\}\right).
\end{equation*}
 \end{proof}

Analysis similar to that in the proof of the previous Theorem shows
that
 \begin{theorem}\label{3.2}Let $\ -\frac{\pi}{2}<\beta< \frac{\pi}{2},\
0<q<1$ and $\ b\neq 0$. If $f$ given by \eqref{in3} belongs to
$\widetilde{\mathcal{ES}}^{\beta}_{q,b}(\phi)$, then
\begin{equation*}
  \left|a_{2}\right| \leq \frac{3bc_{1}}{\varrho[2]_{q}},
\end{equation*}
\begin{equation*}
\left|a_{3}\right| \leq
\frac{5b}{\varrho[3]_{q}}\left(c_{1}+\max\left\{c_{1},
\left|\frac{bc_{1}^{2}}{\varrho}\right| +|c_{2}|\right\}\right),
\end{equation*}
and for any complex number $\mu$,
\begin{equation*}
\left| a_{3} - \mu a^2_{2} \right| \leq
\frac{5b}{\varrho[3]_{q}}\left(c_{1}+\max\left\{c_{1},
\left|\frac{\varrho^{2}[2]_{q}^{2}+9\mu
b}{\varrho^{2}[2]_{q}^{2}}\right|bc_{1}^{2} +|c_{2}|\right\}\right).
\end{equation*}
\end{theorem}

\end{document}